\documentclass[11pt]{amsart}
\usepackage[parfill]{parskip}    % Activate to begin paragraphs with an empty line rather than an indent
\usepackage{fullpage}
\usepackage{graphicx}
\usepackage{amssymb,amsmath,latexsym}
\usepackage{tikz}
\usetikzlibrary{arrows,decorations.markings,decorations.pathreplacing,calc,matrix,intersections}
\tikzset{->-/.style={decoration={markings,mark=at position #1 with {\arrow{>}}},postaction={decorate}}}

\usepackage{epstopdf}
\usepackage{ifpdf}
\usepackage{color}
\usepackage{float}
\usepackage{amscd,amsmath}
\usepackage{amssymb,amsmath,latexsym,color,enumerate,tikz}
\usepackage[all]{xypic}       % Package XYpic of commutative diagram macros
\usepackage[varg]{pxfonts}
\usepackage{amsmath,amsthm,amsfonts,amssymb,fancyhdr,graphics,relsize,tikz-cd,mathtools,faktor,tikz,changepage}
\usepackage{graphicx}
\usepackage{placeins}

\usepackage{dsfont}
\definecolor{red}{rgb}{1,0,0} 

 \definecolor{darkgreen}{rgb}{0, .7, 0}

 \definecolor{purple}{rgb}{.7, 0, 1}

\tikzset{mynode/.style={draw,circle,fill=black,inner sep=2pt,outer sep=0.5pt}}

\newtheorem*{section5}{Theorem~\ref{Theorem Section 5}}
\newtheorem*{section6}{Theorem~\ref{Theorem Section 6}}
\newtheorem*{section7}{Theorem~\ref{Theorem Section 7}}
\newtheorem*{section10}{Theorem~\ref{Theorem Section 10}}
\newtheorem*{theoremC}{Theorem~\ref{Theorem C}}
\newtheorem*{theoremB}{Theorem~\ref{Theorem B}}
\newtheorem*{theoremA}{Theorem~\ref{Theorem A}}
\newtheorem*{theoremend}{Theorem~\ref{Theorem end}}
\newtheorem*{retract}{Theorem~\ref{Theorem Retract}}
\newtheorem*{separability}{Theorem~\ref{Theorem Separability}}
\newtheorem*{theoremconjugacy}{Theorem~\ref{Theorem Conjugacy}}

\newtheorem{theorem}{Theorem}[section]
\newtheorem*{theorem*}{Theorem}
\newtheorem*{lemma*}{Lemma}
\newtheorem{proposition}[theorem]{Proposition}
\newtheorem{lemma}[theorem]{Lemma}
\newtheorem{corollary}[theorem]{Corollary}

\theoremstyle{definition}
\newtheorem{definition}[theorem]{Definition}

\newtheorem{property}[theorem]{Property}
\theoremstyle{remark}
\newtheorem{remark}[theorem]{Remark}

\usepackage{hyperref}
\title{Subgroups of direct products of limit groups over Droms RAAGs}
\author{Jone Lopez de Gamiz Zearra}
\begin{document}
\maketitle

\begin{abstract}
A result of Bridson, Howie, Miller and Short states that if $S$ is a subgroup of type $FP_{n}(\mathbb{Q})$ of the direct product of $n$ limit groups over free groups, then $S$ is virtually the direct product of limit groups over free groups. Furthermore, they characterise finitely presented residually free groups. In this paper these results are generalised to limit groups over Droms right-angled Artin groups. Droms RAAGs are the right-angled Artin groups with the property that all of their finitely generated subgroups are again RAAGs. In addition, we show that the generalised conjugacy problem is solvable for finitely presented groups that are residually a Droms RAAG and that their finitely presentable subgroups are separable.
\end{abstract}
 
\section{Introduction}
\label{Section 1}

In 1984 Baumslag and Roseblade characterized finitely presented subgroups of the direct product of two finitely generated free groups, showing that up to finite index they are themselves a direct product of free groups. This result was generalized in a series of papers by Bridson, Howie, Miller and Short, culminating in a characterization of subgroups of the direct product of finitely many limit groups over free groups, assuming that the subgroups satisfy suitable finiteness properties.

Limit groups over free groups generalize free groups, and they were developed by Kharlmapovich-Miasnikov and Sela as a tool for studying solutions of systems of equations over groups. When Baumslag, Miasnikov and Remeslennikov laid down the foundations of algebraic geometry over groups, they showed that the study of a system of equations over $G$ can be reduced to the study of finitely generated subgroups of the direct product of limit groups over $G$.

The class of right-angled Artin groups (RAAGs) extends the class of finitely generated free groups, by allowing relations saying that some of the generators commute. The present work is the first in a series that aims to study the structure of systems of equations over RAAGs. More precisely, following the work of Bridson, Howie, Miller and Short, the goal is to describe the structure of subgroups of the direct product of limit groups over RAAGs. In this paper, we begin by studying limit groups over Droms RAAGs. Droms RAAGs are those RAAGs with the property that all of their finitely generated subgroups are again RAAGs.

We prove the following theorem:
  
\begin{theoremA}
If $\Gamma_{1},\dots, \Gamma_{n}$ are limit groups over Droms RAAGs and $S$ is a subgroup of $\Gamma_{1}\times \cdots \times \Gamma_{n}$ of type $FP_{n}(\mathbb{Q})$, then $S$ is virtually a direct product of limit groups over Droms RAAGs.
\end{theoremA}

We also prove that if a subgroup of the direct product of $n$ limit groups over Droms RAAGs is not of type $FP_{n}(\mathbb{Q})$, then it has a subgroup with an infinite dimensional homology group:

\begin{theoremB}
Let $\Gamma_{1},\dots, \Gamma_{n}$ be limit groups over Droms RAAGs. Let $S$ be a finitely generated subgroup in $\Gamma_{1}\times \cdots \times \Gamma_{n}$ and set $L_{i}= S \cap \Gamma_{i}$ for $i\in \{1,\dots,n\}$.

If $L_{i}$ is finitely generated for $1 \leq i \leq r$ and not finitely generated for $i>r$, then there is a subgroup of finite index $S_{0}<S$ such that $S_{0}= S_{1}\times S_{2}$, where\\[5pt]
(1) $S_{1}$ is the direct product of limit groups over Droms RAAGs,\\[3pt]
(2) If $n>r$, $H_{k}(S_{2};\mathbb{Q})$ is infinite dimensional for some $k\leq n-r$.
\end{theoremB}

Theorem~\ref{Theorem A} and Theorem~\ref{Theorem B} are the analogues of Theorem A and Theorem B of \cite{something1}. Following the spirit of \cite{something1}, we deduce both theorems from the following result. Recall that a subgroup of a direct product is a \emph{subdirect product} if the projection to each factor is surjective and it is \emph{full} if it intersects each factor non-trivially.

\begin{theoremC}
Let $\Gamma_{1},\dots,\Gamma_{n}$ be limit groups over Droms RAAGs with trivial center and let $S< \Gamma_{1}\times \cdots \times \Gamma_{n}$ be a finitely generated full subdirect product. Then either:\\[3pt]
(1) $S$ is of finite index; or\\[3pt]
(2) $S$ is of infinite index and has a finite index subgroup $S_{0}<S$ such that $H_{j}(S_{0}; \mathbb{Q})$ has infinite dimension for some $j\leq n$.
\end{theoremC}

By \cite{Baumslag}, groups that are residually a Droms RAAG are precisely finitely generated subgroups of the direct product of limit groups over Droms RAAGs, so Theorem~\ref{Theorem A} gives us the following result:

\begin{corollary}
For all $n\in \mathbb{N}$, a group that is residually a Droms RAAG is of type $F_{n}$ if and only if it is of type $FP_{n}(\mathbb{Q})$.
\end{corollary}

In Section~\ref{Section 10}, we study finitely presented groups that are residually a Droms RAAG. In order to state the result, we introduce the following definition: an embedding $S  \hookrightarrow \Gamma_{0}\times \cdots \times \Gamma_{n}$ of a finitely generated group that is residually a Droms RAAG as a full subdirect product of limit groups over Droms RAAGs is \emph{neat} if $\Gamma_{0}$ is abelian (possibly trivial), $S \cap \Gamma_{0}$ is of finite index in $\Gamma_{0}$ and $\Gamma_{i}$ has trivial center for $i\in \{1,\dots,n\}$.

\begin{section10}
Let $S$ be a finitely generated group that is residually a Droms RAAG. The following are equivalent:\\[3pt]
(1) $S$ is finitely presentable;\\[3pt]
(2) $S$ is of type $FP_{2}(\mathbb{Q})$;\\[3pt]
(3) $\dim H_{2}(S_{0};\mathbb{Q})$ is finite for all subgroups $S_{0}<S$ of finite index;\\[3pt]
(4) There exists a neat embedding $S \hookrightarrow \Gamma_{0}\times \cdots \times \Gamma_{n}$ into a product of limit groups over Droms RAAGs such that the image of $S$ under the projection to $\Gamma_{i}\times \Gamma_{j}$ has finite index for $0\leq i<j \leq n$;\\[3pt]
(5) For every neat embedding $S \hookrightarrow \Gamma_{0}\times \cdots \times \Gamma_{n}$ into a product of limit groups over Droms RAAGs the image of $S$ under the projection to $\Gamma_{i}\times \Gamma_{j}$ has finite index for $0\leq i<j \leq n$.
\end{section10}

For any commutative ring $R$, the property $FP_{n}(R)$ is inherited by finite index subgroups and persists in finite extensions. Nevertheless, if $H$ is a subgroup in $G$ and $H$ is of type $FP_{n}(R)$, it does not mean that $G$ is of type $FP_{n}(R)$. For example, $\mathbb{Z}$ is of type $FP_{1}(\mathbb{Z})$ but an infinitely generated free group is not of type $FP_{1}(\mathbb{Z})$. However, when restricted to subgroups of direct products of limit groups over Droms RAAGs, it follows from Theorem~\ref{Theorem Section 10} that any subdirect product of limit groups over Droms RAAGs that contains a finitely presentable full subdirect product is again finitely presentable. We generalize this as follows:

\begin{theoremend}
Let $\Gamma_{1}\times \cdots \times \Gamma_{k}$ be the direct product of limit groups over Droms RAAGs, where $\Gamma_{1}$ is free abelian and $\Gamma_{i}$ is a limit group over a Droms RAAG with trivial center, for $i\in \{2,\dots,k\}$. Let $n\in \mathbb{N} \setminus \{1\}$, $S < \Gamma_{1}\times \cdots \times \Gamma_{k}$ be a full subdirect product and let $T < \Gamma_{1}\times \cdots \times \Gamma_{k}$ be a subgroup that contains $S$. If $S$ is of type $FP_{n}(\mathbb{Q})$, then so is $T$.
\end{theoremend}

In \cite[Section 7]{Bridson3} the authors solve the multiple conjugacy problem for finitely presented residually free groups. Their argument can also be used to prove it for the case of finitely presented residually Droms RAAGs. Thus, in Section~\ref{Section 10} we review their proof and we explain why this can be used also for finitely presented groups that are residually a Droms RAAG. That is, we show:

\begin{theoremconjugacy}
The multiple conjugacy problem is solvable in every finitely presented group that is residually a Droms RAAG.
\end{theoremconjugacy}

Recall that a subgroup $H$ of a group $G$ is \emph{separable} if whenever $g\in G \setminus H$ there is a finite group $Q$ and a homomorphism $f \colon G \mapsto Q$ with $f(g) \notin f(H)$. In addition, $H$ is a \emph{virtual retract} of $G$ if there is a finite index subgroup $G^\prime < G$ such that $H \subseteq G^\prime$ and a homomorphism $r \colon G^\prime \mapsto H$ such that $r_{| H} = \text{id }_{H}$. In \cite{Separability}, Bridson and Wilton, by treating finitely presented residually free groups as subdirect products of limit groups, show that finitely presentable subgroups of residually free groups are separable and that the subgroups of type $FP_{\infty}$ are virtual retracts. Analogously, for finitely presented residually Droms RAAGs the following holds (see Section~\ref{Section 10}):

\begin{retract}
If $G$ is a finitely generated group that is residually a Droms RAAG and $H < G$ is a subgroup of type $FP_{\infty}(\mathbb{Q})$, then $H$ is a virtual retract of $G$.
\end{retract}

\begin{separability}
If $G$ is a finitely generated group that is residually a Droms RAAG and $H$ is a finitely presentable subgroup of $G$ then $H$ is separable in $G$.
\end{separability}

\section{Motivation and background}
\label{Section 2}

\subsection{Algebraic geometry over groups}
\label{Subsection}

Classical algebraic geometry is concerned with the study of algebraic sets, that is solutions of systems of equations $S$ over a ring $R$. Each point in the algebraic set defines a homomorphism from the ring of polynomials $R[x_{1}, \dots, x_{n}]$ to the ring of coefficients $R$. Using this relation, one can establish a one-to-one correspondence between the algebraic set $S$, the radical ideal $Rad(S)$ of the ring $R[x_{1}, \dots, x_{n}]$ and the coordinate ring $R[x_{1}, \dots, x_{n}]\slash Rad(S)$. 

If one requires further conditions on the ring $R$ (namely, if $R$ is Noetherian and an integral domain), then every algebraic set is the union of finitely many irreducible algebraic sets. Irreducible algebraic sets correspond to prime ideals and so to coordinate rings which are integral domains. Therefore, the decomposition of the algebraic set into a finite union of irreducible components is equivalent to the radical ideal being a finite intersection of prime ideals, which in turn corresponds to the coordinate ring being a finitely generated subring of a direct product of coordinate rings of irreducible algebraic sets.

In \cite{Baumslag}, Baumslag, Miasnikov and Remeslennikov lay down the foundations of algebraic geometry over groups, which bears a similarity with classical algebraic geometry. The motivation for doing this comes from the desire to study equations over groups. The authors present group-theoretic counterparts to algebraic sets, coordinate rings, irreducibility and many other concepts. The analog of coordinate ring is called a coordinate group. Furthermore, they establish a correspondence between algebraic sets and coordinate groups. They also characterise coordinate groups of irreducible algebraic sets in terms of their residual properties; more precisely, they show that coordinate groups over a group $G$ are precisely the finitely generated residually $G$ groups. As in classical algebraic geometry, the study of coordinate groups is equivalent to the study of finitely generated subgroups of the direct product of coordinate groups associated to irreducible algebraic sets.

Algebraic geometry over free groups has been developed by Sela and Kharlmapovich-Miasnikov. In \cite{Sela1}, \cite{Sela2}, \cite{Sela3}, Sela uses techniques from geometric group theory, low dimensional topology and Diophantine geometry to give a description of coordinate groups of irreducible algebraic sets over free groups, known as limit groups. Limit groups have an algebraic structure that can be described hierarchically in terms of their JSJ decomposition and from which one can deduce strong properties such as finite presentability.

The general study of coordinate groups over free groups is equivalent to the study of finitely generated subgroups of the direct product of finitely many limit groups. Baumslag and Roseblade started this work in \cite{Roseblade} by describing subgroups of the direct product of two free groups. They explored the difference between finitely generated and finitely presented subgroups. There is no hope of classifying finitely generated subgroups since the isomorphism problem is unsolvable. However, there is a neat structure for finitely presented subgroups: if $F_{1}$ and $F_{2}$ are two finitely generated free groups and $S$ is a finitely presented subgroup in $F_{1}\times F_{2}$, then $S$ is free  or it is virtually the direct product of two free groups.

The study of coordinate groups over free groups was performed by Bridson, Howie, Miller and Short in a series of papers that culminated in \cite{Bridson3} and \cite{something1}. In \cite{something1}, the authors show that, under some finiteness conditions that generalise finite presentability, coordinate groups over free groups (equivalently, residually free groups) have a tame structure. More precisely, if $\Gamma_{1},\dots, \Gamma_{n}$ are limit groups over free groups and $S$ is a subgroup of type $FP_{n}(\mathbb{Q})$ of $\Gamma_{1}\times \cdots \times \Gamma_{n}$, then $S$ is virtually the direct product of limit groups over free groups. The authors also characterize finitely presented coordinate groups in terms of their projections onto pairs $\Gamma_{i} \times \Gamma_{j}$ (see \cite{Bridson3}). Briefly, they prove that if $S$ is a finitely generated coordinate group and $S \hookrightarrow \Gamma_{1}\times \cdots \times \Gamma_{n}$ is an embedding with $\Gamma_{i}$ a limit group over a free group, then $S$ is finitely presented if and only if the image of $S$ under the projection to $\Gamma_{i}\times\Gamma_{j}$ has finite index for $1\leq i < j \leq n$.

These structural results on finitely presented coordinate groups over free groups are useful for studying algorithmic problems. For example, in \cite{Bridson3} it is shown that the generalised conjugacy and the membership problems, although undecidable for general coordinate groups over free groups, are decidable for finitely presented ones.

\subsection{Right-angled Artin groups}

In recent years, RAAGs have been studied extensively especially due to their rich subgroup structure. On the one hand, they are a source of subgroups with interesting finiteness properties. In \cite{Bestvina}, Bestvina and Brady described subgroups of RAAGs that are of type $FP$ but not of type $F$. On the other hand, RAAGs have proved to be fundamental in the study of $3$-manifold groups. Agol showed that non-positively curved cube complexes with hyperbolic fundamental groups are virtually subgroups of RAAGs (see \cite{Agol}). An essential step in the argument is the result of Haglund and Wise that states that fundamental groups of special cube complexes are subgroups of RAAGs (see \cite{Wise}).

Let us recall the definition of a right-angled Artin group. Given a simplicial graph $X$, the corresponding \emph{right-angled Artin group} (\textit{RAAG}), denoted by $GX$, is given by the following presentation. Let $V(X)$ denote the vertex set of $X$. Then,
\[ GX= \langle V(X) \mid \text{if } xy=yx\iff x \text{ and } y \text{ are adjacent}\rangle.\] 
As mentioned above, subgroups of RAAGs can be quite wild, in particular, not all subgroups of RAAGs are themselves RAAGs. Carl Droms provided a condition for a RAAG to have all its subgroups again of this type (see \cite{Droms}): every subgroup of $GX$ is itself a RAAG if and only if no full subgraph of $X$ has either of the forms $C_{4}$ or $P_{4}$ illustrated below:
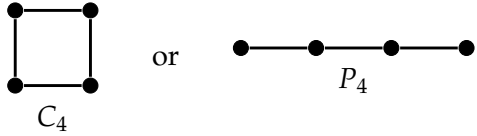
\[
\begin{tikzpicture}
	  \node [mynode] (1) at (0,0) {};
	  \node [mynode] (2) at (1,0) {};
      \node [mynode] (3) at (0,-1)  {};
      \node [mynode] (4) at (1,-1) {};
      \node [label=below:$C_{4}$] (5) at (0.5,-1) {};
      \node [label=above:$\text{or}$] (5) at (2,-1) {};
      \node [mynode] (6) at (3,-0.5) {};
      \node [mynode] (7) at (4,-0.5) {};
      \node [mynode] (8) at (5,-0.5) {};
       \node [mynode] (9) at (6,-0.5) {};
       \node [label=below:$P_{4}$] (5) at (4.5,-0.5) {};
       \draw[line width=1.1pt] (1) -- (2) {};
       \draw[line width=1.1pt] (1) -- (3) {};
       \draw[line width=1.1pt] (2) -- (4) {};
       \draw[line width=1.1pt] (4) -- (3) {};
       \draw[line width=1.1pt] (6) -- (7) {};
       \draw[line width=1.1pt] (7) -- (8) {};
       \draw[line width=1.1pt] (8) -- (9) {};     
\end{tikzpicture}
\]
This type of RAAGs will be called \emph{Droms RAAGs}.

Given a class of groups, there is a natural way to construct other groups using operations such as taking free products or adding center:

\begin{definition}\label{Definition hierarchy}
Let $\mathcal{G}$ be a class of groups. The \emph{$Z \ast$-closure of $\mathcal{G}$}, denoted by $Z \ast (\mathcal{G})$, is the union of classes $(Z \ast (\mathcal{G}))_{k}$ defined as follows. At level $0$, the class $(Z \ast (\mathcal{G}))_{0}$ equals $\mathcal{G}$. A group $G$ lies in $(Z \ast (\mathcal{G}))_{k}$ if and only if
\[ G \cong \mathbb{Z}^m \times (G_{1}\ast \cdots \ast G_{n}),\] where $m\in \mathbb{N}\cup \{0\}$ and the group $G_{i}$ lies in $(Z \ast (\mathcal{G}))_{k-1}$ for all $i\in \{1,\dots,n\}$.
\end{definition}

Notice that if the class $\mathcal{G}$ contains only RAAGs, then so does its $Z \ast$-closure.

\begin{remark}\label{DromsRAAG}
In this terminology, Droms showed in \cite{Droms} that the class of Droms RAAGs is the $Z\ast$-closure of $\mathbb{Z}$.
\end{remark}

\subsection{Limit groups over Droms RAAGs}

If $G$ is a group, a \emph{limit group over $G$} is a group $H$ that is finitely generated and fully residually $G$; that is, for any finite set of non-trivial elements $S \subseteq H$ there is a homomorphism $\varphi \colon H \mapsto G$ which is non-trivial on all $s\in S$.

In the language of model theory, a limit group over $G$ is a finitely generated model of the universal theory of $G$. In algebraic geometry over groups, it is a coordinate group of an irreducible algebraic set over $G$. In universal algebra, limit groups over $G$ are finitely generated subgroups of the ultrapower of $G$. Finally, in topological terms, they are limits of $G$ in a compact space of marked groups. See \cite{limit groups} for a survey of limit groups.

The structure of limit groups over free groups is well understood. Sela characterised them as groups that have a faithful action on a real tree induced by a sequence of homomorphisms to a free group. In addition, one can give a hierarchical description in terms of their non-trivial cyclic JSJ-decomposition. For further reference, see Sela's original papers \cite{Sela1}, \cite{Sela2}, \cite{Sela3} or the introductory notes of Bestvina and Feighn \cite{Feighn}.

In \cite{Montse} and \cite{Montse2}, the authors studied systems of equations over RAAGs. Their results can be used to establish the following properties of limit groups over Droms RAAGs which we will need to prove the results:

\begin{property}\label{Property (1)} Limit groups over Droms RAAGs are finitely presented (see \cite[Corollary 7.8]{Montse}).
\end{property}

\begin{property}\label{Property (2)} Finitely generated subgroups of limit groups over Droms RAAGs are limit groups over Droms RAAGs (see \cite[Theorem 8.1]{Montse}).\end{property}

\begin{property}\label{Property (6)} Limit groups over Droms RAAGs are of type $F_{\infty}$. This follows from \cite[Corollary 9.5]{Montse}.
\end{property}

In \cite{Montse}, Casals-Ruiz, Duncan and Kazachkov define a notion of height of a limit group over a Droms RAAG and prove:

\begin{proposition}[Analogous to Proposition 2.1 in \cite{something1}]\label{Proposition 1}
Let $G$ be a Droms RAAG. Write \[G \cong \mathbb{Z}^n \times G^\prime,\] where $n\in \mathbb{N}\cup \{0\}, G^\prime= G_{1}\ast \cdots \ast G_{r}$ and $G_{i}$ is a Droms RAAG with fewer vertices than $G$, $i\in \{1,\dots,r\}$.

Let $\Gamma$ be a limit group over $G$. If $\text{ht }(\Gamma)= 0$, then $\Gamma$ is a Droms RAAG which is a subgroup of $G$. If $\text{ht }(\Gamma)=h>0$, then
\[ \Gamma \cong \mathbb{Z}^m \times \Gamma^\prime,\] where
\begin{itemize}
\item if $n=0$, then $m=0$;
\item $\Gamma^\prime$ is a limit group over $G^\prime$.
\end{itemize}

Moreover, $\Gamma^\prime$ is a Droms RAAG or $\Gamma^\prime$ admits an acylindrical graph of groups decomposition such that at least one vertex group is a limit group of height $\leq h-1$ and the edge groups are trivial (if $\Gamma^\prime$ admits a non-trivial free product splitting) or infinite cyclic (if $\Gamma^\prime$ is freely-indecomposable).
\end{proposition}

We now define a class of groups that contains limit groups over Droms RAAGs.

\begin{definition}
The class of finitely presented groups $\mathcal{C}$ is defined in a hierarchical manner. It is the union of the classes $\mathcal{C}_{n}$ defined as follows.

At level $0$ we have the class $\mathcal{C}_{0}$ consisting of groups of the form $\mathbb{Z}^n \times (A \ast B)$ where $n\in \mathbb{N}\cup \{0\}$ and $A$ and $B$ are non-trivial finitely presented groups where at least one of $A$ and $B$ has at least cardinality $3$.

A group $\Gamma$ lies in $\mathcal{C}_{n}$ if and only if $\Gamma \cong \mathbb{Z}^m \times \Gamma^\prime$ where $m\in \mathbb{N}\cup \{0\}$ and $\Gamma^\prime$ is a finitely generated acylindrical graph of finitely presented groups, where all of the edge groups are cyclic, at least one of the vertex groups lies in $\mathcal{C}_{n-1}$. Suppose that $\mathbb{Z}^{m^\prime} \times \Gamma^{\prime \prime}$ is such a vertex group that lies in $\mathcal{C}_{n-1}$. Then, if $m=0$, $m^\prime=0$.
\end{definition}

\begin{property}\label{Property (3)} If $\Gamma$ is a limit group over a Droms RAAG, it lies in $\mathcal{C}$.
\end{property}
\begin{proof}
If $\Gamma \cong \mathbb{Z}^n \times \Gamma^\prime$ is a limit group over a Droms RAAG of height $0$ (see Proposition~\ref{Proposition 1}) or if $\Gamma^\prime$ admits a non-trivial free decomposition or if $\Gamma^\prime$ is a Droms RAAG, it clearly lies in $\mathcal{C}$. If it has height $\geq 1$ and $\Gamma^\prime$ is freely-indecomposable, it is an immediate consequence of Proposition~\ref{Proposition 1}.
\end{proof}

Finally, we prove other two properties that will be used in the proofs of the results.

\begin{property}\label{Property (4)}  Limit groups over Droms RAAGs are cyclic subgroup separable.

\end{property}

\begin{proof}
Let us first show that if $G$ is a group, $\mathbb{Z}^m \times G$ is cyclic subgroup separable if and only if $G$ is. Since $G \hookrightarrow \mathbb{Z}^m \times G$, if $\mathbb{Z}^m \times G$ is cyclic subgroup separable, then so is $G$. Now suppose that $G$ is cyclic subgroup separable and $H < \mathbb{Z}^m \times G$ is a cyclic subgroup. Then $H \cong (H \cap G) \times \pi(H)$, where $\pi$ is the projection map $\mathbb{Z}^m \times G \mapsto \mathbb{Z}^m$. The group $H$ is cyclic, so $H \cap G < G$ and $\pi(H) < \mathbb{Z}^m$ are cyclic (or trivial). Let $(g_{1},g_{2})$ be an element of $( \mathbb{Z}^m \times G) \setminus H$. Then, either $g_{1}\in \mathbb{Z}^m \setminus \pi(H)$ or $g_{2}\in G \setminus H \cap G$. In the first case, since $\mathbb{Z}^m$ is cyclic subgroup separable, there is a homomorphism $f \colon \mathbb{Z}^m \mapsto Q$ with $Q$ a finite group. Hence, if we compose the projection map $\pi$ with $f$ we obtain what we want. The second case is symmetric.

Now let $G$ be a Droms RAAG and $\Gamma$ be a limit group over $G$. Let us prove the property by induction on the level in the hierarchy where $G$ first appears (see Remark~\ref{DromsRAAG}). We will check that $\Gamma$ is a subgroup of a group $K$ that can be constructed hierarchically by amalgamating over cyclic subgroups. Then, we will use \cite[Theorem 3.6]{Burillo} to show that $K$ is cyclic subgroup separable. \cite[Theorem 3.6]{Burillo} states that if $A$ and $B$ are quasi-potent and cyclic subgroup separable and $C$ is a virtually cyclic common subgroup, then $A \ast_{C} B$ is cyclic subgroup separable. Finally, since $K$ is cyclic subgroup separable, so is $\Gamma$.

If the level equals $0$, $G$ is $\mathbb{Z}$, so $\Gamma$ is free abelian and so it is cyclic subgroup separable and quasi-potent. Now assume that limit groups over Droms RAAGs of level $n-1$ are cyclic subgroup separable and quasi-potent. Let $G$ be a Droms RAAG of level $n$. By the discussion above, we can assume that $G$ is the free product of Droms RAAGs $G_{i}$ of level at most $n-1$. It follows from \cite{Montse3, Jaligot} that $\Gamma$ is a finitely generated subgroup of an ICE group $K$, where ICE is the smallest class of groups containing all finitely generated free products of limit groups over $G_{i}$ and is closed under amalgamation over cyclic groups. It follows now from \cite[Theorem 3.6]{Burillo} and the induction hypothesis that $K$ is indeed cyclic subgroup separable.
\end{proof}

\begin{property}\label{Property (5)} Let $\Gamma$ be a Droms RAAG with trivial center and let $S$ be a subgroup of $\Gamma$ such that $H_{1}(S; \mathbb{Q})$ is finite dimensional. Then $S$ is finitely generated (and hence a limit group over a Droms RAAG).
\end{property}

\begin{proof}
\cite[Theorem 2]{normalizer} states that if $\Gamma$ is a limit group over a free group and $S$ is a subgroup of $\Gamma$ such that $H_{1}(S; \mathbb{Q})$ is finite $\mathbb{Q}$-dimensional, then $S$ is finitely generated. If we check that \cite[Lemma 4.1]{normalizer} holds also in the case of limit groups over Droms RAAGs, then we can use the same proof as in \cite[Theorem 2]{normalizer} in order to show our property. That is, it suffices to check that if $\Gamma$ has trivial center and it is residually a Droms RAAG and $S< \Gamma$ is a non-cyclic subgroup, then $H_{1}(S ; \mathbb{Q})$ has dimension at least $2$.

Indeed, if $S$ is abelian the result is obvious. If not, there are two elements $s,t\in S$ such that $[s,t]\neq 1$ in $S$ (they do not commute). Since $\Gamma$ is residually a Droms RAAG, there is $\phi \colon \Gamma \mapsto G$ a homomorphism with $G$ a Droms RAAG such that $[\phi(s), \phi(t)]\neq 1$. Then, $\langle \phi(s), \phi(t) \rangle$ is a free group (see \cite[Theorem 1.2]{Baudisch}). Thus $S$ maps onto a non-abelian Droms RAAG and hence onto a free abelian group of rank at least $2$.
\end{proof}

\section{From Theorem~\ref{Theorem C} to Theorem~\ref{Theorem A}}
\label{Section 3}

Let us fix the notation that is used throughout the paper: $S$ is a subgroup of the direct product of limit groups over Droms RAAGs $\Gamma_{i}$, $L_{i}$ is the intersection $S \cap \Gamma_{i}$ and $p_{i}$ is the projection map $\Gamma_{1}\times \cdots \times \Gamma_{n} \mapsto \Gamma_{i}$.

The goal of this section is to show that Theorem~\ref{Theorem A} follows from Theorem~\ref{Theorem C}. Let us recall both theorems:

\begin{theorem}\label{Theorem A}
If $\Gamma_{1},\dots, \Gamma_{n}$ are limit groups over Droms RAAGs and $S$ is a subgroup of $\Gamma_{1}\times \cdots \times \Gamma_{n}$ of type $FP_{n}(\mathbb{Q})$, then $S$ is virtually a direct product of limit groups over Droms RAAGs.
\end{theorem}

\begin{theoremC}
Let $\Gamma_{1},\dots,\Gamma_{n}$ be limit groups over Droms RAAGs with trivial center and let $S< \Gamma_{1}\times \cdots \times \Gamma_{n}$ be a finitely generated full subdirect product. Then either:\\[3pt]
(1) $S$ is of finite index; or\\[3pt]
(2) $S$ is of infinite index and has a finite index subgroup $S_{0}<S$ such that $H_{j}(S_{0}; \mathbb{Q})$ has infinite dimension for some $j\leq n$.
\end{theoremC}

We first reduce Theorem~\ref{Theorem A} to a simpler case:

\begin{proposition}\label{Proposition 2}
If Theorem~\ref{Theorem A} holds under the below assumptions (1)--(5), Theorem~\ref{Theorem A} holds in general:\\[5pt]
(1) $n\geq 2$.\\[3pt]
(2) Each projection map $p_{i} \colon S \mapsto \Gamma_{i}$ is surjective.\\[3pt]
(3) Each intersection $L_{i}= S \cap \Gamma_{i}$ is non-trivial.\\[3pt]
(4) Each $\Gamma_{i}$ has trivial center.\\[3pt]
(5) Each $\Gamma_{i}$ splits as a HNN extension over an infinite cyclic subgroup $C_{i}$ or the trivial group with stable letter $t_{i}\in L_{i}$.
\end{proposition}

\begin{proof}
(1) If $n=1$, $S< \Gamma_{1}$ has type $FP_{1}(\mathbb{Q})$, so it is finitely generated. Therefore, it is a limit group over a Droms RAAG.\\[3pt]
(2) Since $S$ has type $FP_{n}(\mathbb{Q})$ for some $n\geq 2$, $S$ is in particular finitely generated. Hence, $p_{i}(S)$ is also finitely generated; in particular, it is a limit group over a Droms RAAG. Therefore, we can replace each $\Gamma_{i}$ by $p_{i}(S)$.\\[3pt]
(3) Assume that there is a $L_{i}$, say $L_{n}$, which is trivial. Then, the projection map $q_{n} \colon S \mapsto \Gamma_{1}\times \cdots \times \Gamma_{n-1}$ is injective, so \[S \cong q_{n}(S) < \Gamma_{1} \times \cdots \times \Gamma_{n-1}.\] After iterating this argument, we may assume that each $L_{i}$ is non-trivial.\\[3pt]
(4) Assume that $\Gamma_{i} \cong \mathbb{Z}^{m_{i}}\times \Gamma_{i}^\prime$ for $m_{i}\in \mathbb{N} \cup \{0\}$ and $\Gamma_{i}^\prime$ is a limit group over a Droms RAAG with trivial center, $i\in \{1,\dots,n\}$. Then, \[\Gamma_{1}\times \cdots \times \Gamma_{n} \cong \mathbb{Z}^{m_{1}+\dots + m_{n}} \times (\Gamma_{1}^\prime \times \cdots \times \Gamma_{n}^\prime).\] Then, $S$ is isomorphic to $ S \cap (\Gamma_{1}^\prime \times \cdots \times \Gamma_{n}^\prime) \times \pi(S)$, where $\pi$ denotes the projection map \[ \Gamma_{1}\times \cdots \times \Gamma_{n} \mapsto \mathbb{Z}^{m_{1}+\dots + m_{n}}.\] But then $ S \cap (\Gamma_{1}^\prime \times \cdots \times \Gamma_{n}^\prime)$ is of type $FP_{n}(\mathbb{Q})$. Since $\pi(S)$ is a limit group over a Droms RAAG, if Theorem~\ref{Theorem A} holds in the case where all the $\Gamma_{i}$ have trivial center, then $S$ is virtually a direct product of limit groups over Droms RAAGs.\\[3pt]
(5) The subgroup $L_{i}$ of $\Gamma_{i}$ is normal by (2) and non-trivial by (3). If we denote by $T$ the Bass-Serre tree corresponding to the splitting described in Proposition~\ref{Proposition 1}, by \cite[Corollary 2.2]{Bridson3} $L_{i}$ contains a hyperbolic isometry $t_{i}$. Then, by \cite[Theorem 3.1]{Bridson3}, there is $\Lambda_{i}$ a finite index subgroup in $\Gamma_{i}$ such that $\Lambda_{i}$ is a HNN extension with stable letter $t_{i}$ and with trivial or cyclic edge stabilizer.

We can replace each $\Gamma_{i}$ by the subgroup of finite index $\Lambda_{i}$ and $S$ by $T= S \cap (\Lambda_{1}\times \cdots \times \Lambda_{n})$. Then, $T$ has finite index in $S$, so in order to prove Theorem~\ref{Theorem A} it suffices to show that $T$ is virtually the direct product of limit groups over Droms RAAGs.
\end{proof}
Finally, assume that Theorem~\ref{Theorem C} holds and let us prove Theorem~\ref{Theorem A}. Assume that the group $S$ is as in Theorem~\ref{Theorem A} with the additional assumptions of Proposition~\ref{Proposition 2}. By Theorem~\ref{Theorem C}, $S$ has finite index in $\Gamma_{1}\times \cdots \times \Gamma_{n}$ or $S$ has infinite index and there is a subgroup $S_{0}< S$ of finite index in $S$ such that $H_{j}(S_{0};\mathbb{Q})$ has infinite dimension for some $j\in \{0,\dots,n\}$. In particular, $S_{0}$ is not of type $FP_{j}(\mathbb{Q})$, and since $S_{0}$ has finite index in $S$, this implies that $S$ is not of type $FP_{j}(\mathbb{Q})$. Thus, $S$ is not of type $FP_{n}(\mathbb{Q})$.

\section{Organization of the proof of Theorem~\ref{Theorem C}}
\label{Section 4}

We have shown in the previous section that Theorem~\ref{Theorem A} follows from Theorem~\ref{Theorem C}. We now focus on proving Theorem~\ref{Theorem C} so the properties of Proposition~\ref{Proposition 2} can be assumed. The proof extends from Section~\ref{Section 5} to Section~\ref{Section 8}.

In Section~\ref{Section 5} we prove the following generalisation of the result that non-trivial, finitely generated normal subgroups in free products have finite index.

\begin{section5}
Let $\Gamma$ be a group in the class $\mathcal{C}$ with trivial center and suppose $\Gamma$ has subgroups $1 \neq N < G < \Gamma$ with $N$ normal in $\Gamma$ and $G$ finitely generated. Then $| \Gamma \colon G | < \infty$.
\end{section5}

In section 6 we use this result and Proposition~\ref{Proposition 2} (5) to deduce the following useful fact (see Section~\ref{Section 6}):

\begin{section6}
Let $\Gamma_{1},\dots, \Gamma_{n}$ be limit groups over Droms RAAGs with the properties of Proposition~\ref{Proposition 2}. Suppose that $H_{2}(S_{1}; \mathbb{Q})$ is finite dimensional for all $S_{1}<S$ finite dimensional. Then, $\Gamma_{j} \slash L_{j}$ is virtually nilpotent.
\end{section6}

We also state some properties of virtually nilpotent groups in Section~\ref{Section 6} that will be used in the proof of Theorem~\ref{Theorem C}.

In Section~\ref{Section 7} we show that there is a subgroup $N_{0}$ in $\Gamma_{1}\times \cdots \times \Gamma_{n}$ with infinite dimensional $k$-th homology for some $k\in \{0,\dots,n\}$. More precisely, we prove the following: 

\begin{section7}
Let $\Gamma_{1},\dots, \Gamma_{n}$ be limit groups over Droms RAAGs with trivial center, and $N$ the kernel of an epimorphism $\Gamma_{1}\times \cdots \Gamma_{n} \mapsto \mathbb{Z}$. Then, there is a subgroup of finite index $N_{0}\subseteq N$ such that at least one of the homology groups $H_{k}(N_{0}; \mathbb{Q})$ has infinite dimension.
\end{section7}

Finally, Section~\ref{Section 8} is devoted to the proof of Theorem~\ref{Theorem C} which uses an inductive argument and the lower central series obtained from the fact that $\Gamma_{j} \slash L_{j}$ is virtually nilpotent.

\section{Proof of Theorem~\ref{Theorem Section 5}}
\label{Section 5}

In this section we prove Theorem~\ref{Theorem Section 5}.

\begin{theorem}\label{Theorem Section 5}
Let $\Gamma$ be a group in the class $\mathcal{C}$ with trivial center and suppose $\Gamma$ has subgroups $1 \neq N < G < \Gamma$ with $N$ normal in $\Gamma$ and $G$ finitely generated. Then $| \Gamma \colon G | < \infty$.
\end{theorem}

\begin{proof}
Let $\Gamma \in \mathcal{C}$ with trivial center. Then, $\Gamma$ acts non-trivially, acylindrically, cocompactly and minimally on a tree $T$ with trivial or cyclic edge stabilizers. By \cite[Lemma 5.1]{something1}, $G \backslash T$ is finite. If the edge stabilizers are trivial, this implies that $| \Gamma \colon G|$ is finite because the number of edges in $G \backslash T$ is an upper bound for that number.

If the edge stabilizers are cyclic, $|G \backslash \Gamma \slash \Gamma_{e}|$ is finite for the stabilizer $\Gamma_{e}$ of any edge $e$ in $T$. It remains to show that $|\Gamma \colon G|$ is finite (see Proposition~\ref{Proposition Section 5} below).
\end{proof}

\begin{proposition}\label{Proposition Section 5}
Let $\Gamma \in \mathcal{C}$ be with trivial center, and let $D,G$ be subgroups of $\Gamma$ with $D$ cyclic and $G$ finitely generated. If $|G \backslash \Gamma \slash D| < \infty$, then $| \Gamma \colon G |< \infty$.
\end{proposition}

\begin{proof}
Let us prove it by induction on the level $l=l(\Gamma)$ in the hierarchy of $\mathcal{C}= \cup_{n} \mathcal{C}_{n}$ where $\Gamma$ first appears.

If $l=0$, $\Gamma$ is a finitely presented free product, so it has a non-trivial, acylindrical, cocompact action on a tree $T$ with trivial edge stabilizers. If $l>0$, $\Gamma$ has a non-trivial, acylindrical, cocompact action on a tree $T$ with infinite cyclic edge stabilizers and the stabilizer of some vertex $w$ is in $\mathcal{C}_{n-1}$.

Let $c$ be a generator of the cyclic subgroup $D$. We distinguish two cases depending on whether $c$ is elliptic or hyperbolic.

Suppose that $c$ is elliptic and that it fixes a vertex $t$ of $T$. The proof that $X = G \backslash T$ is finite is the same proof as in \cite[Proposition 5.2]{something1}:

First, we show that $X$ has finite diameter. Since $| G \backslash \Gamma \slash D|$ is finite, there are $\gamma_{1},\dots,\gamma_{n}\in \Gamma$ such that \[ \Gamma= G \gamma_{1} D \dot{\cup} \cdots \dot{\cup} G \gamma_{n} D.\] Since $c \cdot t=t$, the $\Gamma$-orbit of $t$ consists of finitely many $G$-orbits, namely $G(\gamma_{i} \cdot t)$, $i\in \{1,\dots,n\}$. Since the action of $\Gamma$ on $T$ is cocompact, there is $m>0$ such that $T$ is the $m$-neighbourhood of $\Gamma t$. Therefore, $X$ has finite diameter.

Second, $\pi_{1}(X)$ has finite rank. Since $G$ is the graph of groups of $X$, there is $r \colon G \mapsto \pi_{1}(X)$ a retract. Since $G$ is finitely generated, $r(G)= \pi_{1}(X)$ is also finitely generated.

Third, $X$ has finitely many valency $1$ vertices. Suppose that there are infinitely many valency $1$ vertices. Then there is $\overline{u}= Gu$ a vertex with valency $1$ where $u$ is a vertex in $T$ and $\overline{e}= Ge$ is the unique edge in $X$ incident at $\overline{u}$ ($e$ is an edge in $T$ incident at $u$) such that $G_{\overline{u}}= G_{\overline{e}}$, where $G_{\overline{u}}$ and $G_{\overline{e}}$ denote the stabilizers of $u$ and $e$ in $G$, respectively. Otherwise, $G$ would not be finitely generated. Since $G$ acts by isometries in $T$ and $\overline{u}=Gu$ has valency $1$ in $X$, $G_{\overline{u}}$ acts transitively on the link of $u$, $\text{lk}(u)$. Hence, $ |\text{lk}(u)|=|G_{\overline{u}} \colon G_{\overline{e}}|=1$, so $u$ is a valency $1$ vertex of $T$. This is a contradiction since $T$ is minimal as a $\Gamma$-tree.

In conclusion, $X$ has finite diameter, finite rank and finitely many valency $1$ vertices, so $X$ is a finite graph.

Suppose that $l=0$. Then, $|G \backslash \Gamma|= | \Gamma \colon G|$ is finite because the number of edges in $X$ is an upper bound for that number. If $l>0$ there is a vertex $w$ such that its stabilizer, $\Gamma_{w}$, is in $\mathcal{C}_{l-1}$. Recall that by the definition of the class $\mathcal{C}$, $\Gamma_{w}$ does not have center. The number $| (G \cap \Gamma_{w}) \backslash \Gamma_{w} \slash \Gamma_{e}|$ is finite because it is bounded above by the number of edges in $X$ incident at $Gw$ that are images of edges $\gamma e \in \Gamma e$. So by inductive hypothesis $G \cap \Gamma_{w}$ has finite index in $\Gamma_{w}$. Define the action $\alpha \colon  \Gamma_{w} \times G \backslash \Gamma \mapsto G \backslash \Gamma$ by \[ \alpha((x, G \gamma))= G( \gamma x).\] Then, \[ G \backslash \Gamma= \dot{\bigcup}_{\gamma \in G \backslash \Gamma \slash \Gamma_{w}} \text{Orbit}(G \gamma).\] Note that $|G \backslash \Gamma \slash \Gamma_{w}|< \infty$ because it is bounded above by the number of edges in $X$. Moreover, by the orbit-stabilizer theorem, $\text{Orbit}(G \gamma)$ is in bijection with $\Gamma_{w} \slash ( \gamma^{-1} G \gamma \cap \Gamma_{w})$ which is finite. In conclusion, $G \backslash \Gamma$ is finite.

The argument for the case when $c$ acts hyperbolically on $T$ is the same as in \cite[Proposition 5.2]{something1}.
\end{proof}

\section{Proof of Theorem~\ref{Theorem Section 6} and nilpotent quotients}
\label{Section 6}

From now on, we shall restrict to the conditions of Theorem~\ref{Theorem A} with the additional assumptions of Proposition~\ref{Proposition 2}. That is, each projection map $p_{i} \colon S \mapsto \Gamma_{i}$ is surjective, each $\Gamma_{i}$ has trivial center and each $\Gamma_{i}$ splits as a HNN extension over an infinite cyclic subgroup (or the trivial group) with stable letter $t_{i}\in L_{i}$.

The goal of this section is to prove that $\Gamma_{i} \slash L_{i}$ is virtually nilpotent. For the convenience of the reader, we recall the proof of Bridson, Howie, Miller and Short in \cite{something1}, pointing out which properties of limit groups over free groups are used and checking that these are all satisfied by limit groups over Droms RAAGs.

We introduce the following notation. We write $K_{i}$ for the kernel of the $i$-th projection map $p_{i}$ and $N_{i,j}$ for the image of $K_{i}$ under the $j$-th projection map $p_{j}$. That is,
\[ K_{i}= \{ (s_{1},\dots,s_{i-1},1,s_{i+1},\dots,s_{n}) \in S \},\]
\[ N_{i,j}= \{s_{j}\in \Gamma_{j} \mid (\ast,\dots,\ast,1,\ast,\dots,\ast,s_{j},\ast,\dots,\ast)\in S\}.\]
Note that $N_{i,j}$ is a normal subgroup in $\Gamma_{j}$ because $p_{j}$ is surjective.

\begin{lemma}{\cite[Lemma 6.1]{something1}}
The iterated commutator $[N_{1,j},N_{2,j},\dots, N_{j-1,j},N_{j+1,j},\dots, N_{n,j}]$  is contained in $L_{j}$.
\end{lemma}

Let $\Gamma_{i}= B_{i} \ast_{C_{i}}$ be the HNN splitting of $\Gamma_{i}$ with stable letter $t_{i}\in L_{i}$ and $C_{i}$ an infinite cyclic subgroup (or the trivial group). Let us consider, for $i\neq j$, the group $A_{i,j}= p_{j}({p_{i}}^{-1}(C_{i}))$. Since $N_{i,j}= p_{j}({p_{i}}^{-1}(\{1\}))$, $N_{i,j}<A_{i,j}<\Gamma_{j}$. The strategy is to show that $N_{i,j}$ has finite index in $\Gamma_{j}$. This is achieved in two steps. We sketch the proof below:\\[5pt]
(a) If $H_{2}(S;\mathbb{Q})$ is finite dimensional, $A_{i,j}$ has finite index in $\Gamma_{j}$ and $A_{i,j} \slash N_{i,j}$ is cyclic (see \cite[Lemma 6.2]{something1}).

This step is proved as follows. By using the HNN splitting of $\Gamma_{j}$, we obtain a decomposition of $S$ as a HNN extension. An argument using the Mayer-Vietoris sequence for HNN extensions implies that $H_{1}(A_{i,j}; \mathbb{Q})$ is finite dimensional. In the case of limit groups over free groups, it follows from \cite[Theorem 2]{normalizer} that $A_{i,j}$ is finitely generated. In the case of limit groups over Droms RAAGs, it follows from Property~\ref{Property (5)}. The subgroup $A_{i,j}$ contains the non-trivial normal subgroup $L_{j}$, so Theorem~\ref{Theorem Section 5} implies that $A_{i,j}$ has finite index in $\Gamma_{j}$.\\[5pt]
(b) If $H_{2}(S_{1}; \mathbb{Q})$ is finite dimensional for each subgroup $S_{1}$ of finite index in $S$, $N_{i,j}$ has actually finite index in $\Gamma_{j}$ (\cite[Proposition 6.3]{something1}).

This step is shown with homological arguments. The only property used of limit groups over free groups is the splitting of $\Gamma_{i}$ as a HNN extension. Assuming that $\Gamma_{j} \slash N_{i,j}$ is virtually cyclic, it is shown that there is $S_{1}<S$ a finite index subgroup such that $H_{2}(S_{1};\mathbb{Q})$ is infinite dimensional, which is a contradiction.

Theorem~\ref{Theorem Section 6} now follows easily.

\begin{theorem}\label{Theorem Section 6}
Let $\Gamma_{1},\dots, \Gamma_{n}$ be limit groups over Droms RAAGs with the properties of Proposition~\ref{Proposition 2}. Suppose that $H_{2}(S_{1}; \mathbb{Q})$ is finite dimensional for all $S_{1}<S$ finite dimensional. Then, $\Gamma_{j} \slash L_{j}$ is virtually nilpotent.
\end{theorem}

\begin{proof}
Let $G$ be a group, $L$ a normal subgroup in $G$ and $N_{1},\dots, N_{n}$ normal subgroups of finite index in $G$ such that $[N_{1},\dots, N_{n}]\subseteq L$. Then, if we denote $N_{1}\cap \dots \cap N_{n}$ by $N$, $N$ has finite index in $G$ and $[N,\dots,N]$ is a subgroup of $L$. Now $N$ is normal in $G$, and accordingly $[NL,\dots,NL]$ is also a subgroup in $L$. Therefore, $NL\slash L$ is nilpotent and it has finite index in $G \slash L$ because $N$ has finite index in $G$.
\end{proof}

We finally state two results about finitely generated virtually nilpotent groups that will be used in the proof of Theorem~\ref{Theorem A}.

\begin{lemma}{\cite[Lemma 8.1]{something1}}\label{Nilpotent 1}
Let $G$ be a finitely generated virtually nilpotent group and let $H$ be a subgroup of infinite index. Then there exists a subgroup $K$ of finite index in $G$ and an epimorphism $f \colon K \mapsto \mathbb{Z}$ such that $( H \cap K) \subseteq \text{ker }(f)$.
\end{lemma}

Repeated applications of the previous lemma yield the following result.

\begin{corollary}{\cite[Corollary 8.2]{something1}}\label{Nilpotent 2}
Let $G$ be a finitely generated, virtually nilpotent group and let $H$ be a subgroup of $G$. Then there is a subnormal chain \[ H_{0} < H_{1} < \cdots < H_{r}=G,\] where $H_{0}$ is a subgroup of finite index in $H$ and for each $i$ the quotient group $H_{i+1} \slash H_{i}$ is either finite or cyclic.
\end{corollary}

\section{Proof of Theorem~\ref{Theorem Section 7}}
\label{Section 7}

The aim of this section is to prove Theorem~\ref{Theorem Section 7}:

\begin{theorem}\label{Theorem Section 7}
Let $\Gamma_{1},\dots, \Gamma_{n}$ be limit groups over Droms RAAGs with trivial center, and $N$ the kernel of an epimorphism $\Gamma_{1}\times \cdots \Gamma_{n} \mapsto \mathbb{Z}$. Then, there is a subgroup of finite index $N_{0}\subseteq N$ such that at least one of the homology groups $H_{k}(N_{0}; \mathbb{Q})$ has infinite dimension.
\end{theorem}

Recall that by Property~\ref{Property (6)} limit groups over Droms RAAGs are of type $F_{\infty}$. Thus the following result can be used in our case:

\begin{proposition}{\cite[Proposition 7.1]{something1}}\label{Prev Proposition}
If $\Gamma_{1},\dots,\Gamma_{n}$ are groups of type $FP_{n}(\mathbb{Z})$ and $\phi \colon \Gamma_{1}\times \cdots \times \Gamma_{n} \mapsto \mathbb{Z}$ has non-trivial restriction to each factor, then $H_{j}(\text{ker }\phi; \mathbb{Z})$ is finitely generated for $j\leq n-1$.
\end{proposition}

By using that result, in \cite[Theorem 7.2]{something1} they show that if $\Gamma_{1},\dots, \Gamma_{n}$ are non-abelian limit groups over free groups and $S$ is the kernel of an epimorphism $\Gamma_{1}\times \cdots \times \Gamma_{n} \mapsto \mathbb{Z}$ such that the restriction to each of the $\Gamma_{i}$ is an epimorphism, then $H_{n}(S;\mathbb{Q})$ has infinite dimension. 

The same statement holds for limit groups over Droms RAAGs with trivial center, namely,

\begin{theorem}\label{Theorem Section 7.2}
Let $\Gamma_{1},\dots, \Gamma_{n}$ be limit groups over Droms RAAGs with trivial center and let $S$ be the kernel of an epimorphism $\phi \colon \Gamma_{1} \times \cdots \times \Gamma_{n} \mapsto \mathbb{Z}$. If the restriction of $\phi$ to each of the $\Gamma_{i}$ is an epimorphism, then $H_{n}(S;\mathbb{Q})$ has infinite dimension.
\end{theorem}

The proof is an adaptation of the proof from \cite{something1} to the setting of limit groups over Droms RAAGs:

\begin{proof}
The proof is by induction on $n$. If $n=1$, the group $S=\text{ker }\phi$ is a normal subgroup of a limit group over a Droms RAAG with trivial center. If $H_{1}(S;\mathbb{Q})$ were finite dimensional, $S$ would be finitely generated by Property~\ref{Property (5)}. Then, $S$ would have finite index in $\Gamma_{1}$ (Theorem~\ref{Theorem Section 5}) and this is a contradiction because $\phi$ is an epimorphism.

If $n>2$, let us consider the Lyndon–-Hochschild–-Serre spectral sequence for
\[ 
\begin{tikzcd}
1 \arrow{r}  & S_{n-1}= \text{ker }(p_{n}) \arrow{r} & S \arrow{r}{p_{n}} & \Gamma_{n} \arrow{r} & 1.
\end{tikzcd} 
\]
By Proposition~\ref{Prev Proposition}, $H_{j}(S;\mathbb{Q})$ is finite dimensional for $j<n$, so in order to prove that $H_{n}(S; \mathbb{Q})$ is infinite dimensional, it suffices to show that $ H_{1}(\Gamma_{n}; H_{n-1}(S_{n-1}; \mathbb{Q}))$ is infinite dimensional. By using theory of modules over a PID and Shapiro's Lemma, it can be shown that $ H_{1}(\Gamma_{n}; H_{n-1}(S_{n-1}; \mathbb{Q}))$ is the direct sum of two homology groups, say $A$ and $B$, and that $B \cong H_{1}(L_{n}; \mathbb{Q})$. Recall that $L_{n}= S \cap \Gamma_{n}$, so it is the kernel of the epimorphism $\phi_{| \Gamma_{n}}$. Thus, as it has infinite index in $\Gamma_{n}$, $L_{n}$ is infinitely generated (Theorem~\ref{Theorem Section 5}). In conclusion, $H_{1}(L_{n}; \mathbb{Q})$ is infinite dimensional by Property~\ref{Property (5)}.
\end{proof}

Let us end the section by proving Theorem~\ref{Theorem Section 7}.

\begin{proof}[Proof of Theorem~\ref{Theorem Section 7}]
Suppose that the restriction of $\phi$ to some $\Gamma_{i}$ is trivial, for example to $\Gamma_{1}$. Then, $\phi \colon \Gamma_{2}\times \cdots \times \Gamma_{n} \mapsto \mathbb{Z}$ is an epimorphism. Thus, we may assume that the restriction of $\phi$ to each of the factors is non-trivial. Then, $\phi(\Gamma_{i})=m\mathbb{Z}$ for some non-zero integer $m$. We may replace $\Gamma_{i}$ by $\Delta_{i}= \phi^{-1}(m\mathbb{Z})$. Note that $\Delta_{i}$ has finite index in $\Gamma_{i}$ and that the restriction of $\phi$ to $\Delta_{1}\times \cdots \times \Delta_{n}$ satisfies the conditions of Theorem~\ref{Theorem Section 7.2}.
\end{proof}

\section{Proof of Theorem~\ref{Theorem C}}
\label{Section 8}

Let us recall the statement of Theorem~\ref{Theorem C}:

\begin{theorem}\label{Theorem C}
Let $\Gamma_{1},\dots,\Gamma_{n}$ be limit groups over Droms RAAGs with trivial center and let $S< \Gamma_{1}\times \cdots \times \Gamma_{n}$ be a finitely generated full subdirect product. Then either:\\[3pt]
(1) $S$ is of finite index; or\\[3pt]
(2) $S$ is of infinite index and has a finite index subgroup $S_{0}<S$ such that $H_{j}(S_{0}; \mathbb{Q})$ has infinite dimension for some $j\leq n$.
\end{theorem}

\begin{proof}
Let $\Gamma= \Gamma_{1} \times \cdots \times \Gamma_{n}$ and $L=L_{1}\times \cdots \times L_{n}$. We only need to consider the case when $S$ has infinite index in $\Gamma$. By contradiction, suppose that for all the subgroups $S_{0}$ of finite index in $S$ and for all $0\leq j \leq n$, $H_{j}(S_{0};\mathbb{Q})$ has finite dimension.

By Theorem~\ref{Theorem Section 6}, $\Gamma_{i}/L_{i}$ is virtually nilpotent, for all $i\in \{1,\dots, n\}$. Thus, $\Gamma/L$ is virtually nilpotent. By applying Lemma~\ref{Nilpotent 1} to $\Gamma /L$, there is a finite index subgroup $\Lambda$ in $\Gamma$ containing $L$ and an epimorphism $f\colon \Lambda / L \mapsto \mathbb{Z}$ such that $\Lambda / L \cap S / L \subseteq \text{ker }f.$ Define $g$ to be $f \circ \pi$, where $\pi \colon \Lambda \mapsto \Lambda /L$ is the projection map. By definition, $g$ is an epimorphism and $\Lambda \cap S$ is contained in $\text{ker }g$.

We replace $S$ by $S\cap \Lambda$. Let $p_{i}\colon \Gamma \mapsto \Gamma_{i}$ be the $i$-th projection map. Then, we also replace $\Gamma_{i}$ by $p_{i}(\Lambda)$. By using this argument we ensure that $L\subseteq S \subseteq N$, where $N$ is the kernel of an epimorphism $\Lambda \mapsto \mathbb{Z}$. By Theorem~\ref{Theorem Section 7}, there is a finite index subgroup $N_{0}$ in $N$ and $j\leq n$ such that $H_{j}(N_{0};\mathbb{Q})$ has infinite dimension. By Corollary~\ref{Nilpotent 2}, there is a subgroup $S_{0}$ contained in $S\cap N_{0}$, which has finite index in $S$ and a series $S_{0} \triangleleft S_{1}\triangleleft \cdots \triangleleft S_{k}=N_{0}$ with $S_{i+1}/S_{i}$ finite or cyclic for each $i$.

We now use the following lemma to reach a contradiction.

\begin{lemma}\cite[Lemma 8.3]{something1}
Let $S_{0}$ be a normal subgroup in $S_{1}$ with $S_{1} /S_{0}$ finite or cyclic. If $H_{j}(S_{0};\mathbb{Q})$ is finite dimensional for $0\leq j \leq n$, then $H_{j}(S_{1};\mathbb{Q})$ is finite dimensional for $0\leq j \leq n$.
\end{lemma}

\begin{proof}
Consider the Lyndon–-Hochschild–-Serre spectral sequence in homology for the short exact sequence
\[ 
\begin{tikzcd}
1 \arrow{r}  &S_{0} \arrow{r} & S_{1} \arrow{r} & S_{1}/S_{0} \arrow{r} & 1.
\end{tikzcd} 
\]
It is known (see, for instance, \cite[Chapter 3, Corollary 10.2]{something5} and \cite[Chapter 8, Exercise 6.6]{something5}) that $S_{1}/S_{0}$ has homological dimension $0$ if it is finite, and $1$ if it is infinite cyclic. Hence, $E^2_{p,q}=0$ for all $p\geq 2$. Therefore, the spectral sequence stabilizes at the $E^2$ page, so \[\dim{H_{j}(S_{1};\mathbb{Q})}=\dim{H_{0}(S_{1}/S_{0};H_{j}(S_{0};\mathbb{Q}))}+\dim{H_{1}(S_{1}/S_{0};H_{j-1}(S_{0};\mathbb{Q}))}.\]
Since $S_{1}/S_{0}$ is of type $\text{FP}_{1}$, both dimensions are finite.
\end{proof}

Coming back to the main proof, by repeatedly applying this lemma, we obtain that $H_{j}(N_{0};\mathbb{Q})$ has finite dimension for all $j\leq n$, contradicting Theorem~\ref{Theorem Section 7}.
\end{proof}

\section{Proof of Theorem~\ref{Theorem B}}
\label{Section 9}

In this section we prove Theorem~\ref{Theorem B}. Let us first recall the statement:

\begin{theorem}\label{Theorem B}
Let $\Gamma_{1},\dots, \Gamma_{n}$ be limit groups over Droms RAAGs. Let $S$ be a finitely generated subgroup in $\Gamma_{1}\times \cdots \times \Gamma_{n}$ and set $L_{i}= S \cap \Gamma_{i}$ for $i\in \{1,\dots,n\}$.

If $L_{i}$ is finitely generated for $1 \leq i \leq r$ and not finitely generated for $i>r$, then there is a subgroup of finite index $S_{0}<S$ such that $S_{0}= S_{1}\times S_{2}$, where\\[5pt]
(1) $S_{1}$ is the direct product of limit groups over Droms RAAGs,\\[3pt]
(2) If $n>r$, $H_{k}(S_{2};\mathbb{Q})$ is infinite dimensional for some $k\leq n-r$.
\end{theorem}

We first reduce to the case that the limit groups $\Gamma_{i}$ have trivial center. Assume that for $i\in \{1,\dots, n\}$, $\Gamma_{i}= \mathbb{Z}^{m_{i}} \times \Gamma_{i}^\prime$ with $m_{i}\in \mathbb{N} \cup \{0\}$ and $\Gamma_{i}^\prime$ is a limit group over a Droms RAAG with trivial center. Then, \[ \Gamma_{1}\times \cdots \times \Gamma_{n}= \mathbb{Z}^{m_{1}+\dots+m_{n}} \times (\Gamma_{1}^\prime\times \cdots \times \Gamma_{n}^\prime),\]
\[ S \cong (S \cap \mathbb{Z}^{m_{1}+\dots+m_{n}}) \times \pi(S),\] where $\pi$ is the projection map $\mathbb{Z}^{m_{1}+\dots+m_{n}} \times (\Gamma_{1}^\prime\times \cdots \times \Gamma_{n}^\prime) \mapsto \Gamma_{1}^\prime\times \cdots \times \Gamma_{n}^\prime.$

Since $S \cap \mathbb{Z}^{m_{1}+\dots+m_{n}}$ is a limit group over a Droms RAAG, it suffices to work with the finitely generated subgroup $\pi(S) < \Gamma_{1}^\prime\times \cdots \times \Gamma_{n}^\prime$.

Let us show that for $i\in \{1,\dots,n\}$, $S \cap \Gamma_{i}$ is finitely generated if and only if $\pi(S) \cap \Gamma_{i}^\prime$ is finitely generated. In order to distinguish both groups, we denote them by ${L_{i}}^S$ and ${L_{i}}^{\pi(S)}$, respectively. We prove it for $i=1$.

By definition, \[{L_{1}}^{\pi(S)}= \pi(S) \cap \Gamma_{1}^\prime= \{s\in \Gamma_{1}^\prime \mid (sz_{1},z_{2},\dots,z_{n})\in S, z_{i}\in \mathbb{Z}^{m_{i}}\},\]
\[\pi({L_{1}}^S)= \pi(S \cap \Gamma_{1})= \{ s\in \Gamma_{1}^\prime \mid (sz_{1},1,\dots,1)\in S, z_{1}\in \mathbb{Z}^{m_{1}} \}.\]
We claim that ${L_{1}}^{\pi(S)}= \pi({L_{1}}^S) \times A$, where $A$ is an abelian group. Let us first check that \[{L_{1}}^{\pi(S)} \slash \pi({L_{1}}^S) \cong p(S \cap (\Gamma_{1} \times \mathbb{Z}^{m_{2}} \times \cdots \times \mathbb{Z}^{m_{n}})) \slash p(S \cap (\mathbb{Z}^{m_{1}} \times \cdots \times \mathbb{Z}^{m_{n}})),\] where $p$ is the projection homomorphism $p\colon \Gamma_{1}\times \cdots \times \Gamma_{n} \mapsto \Gamma_{2}\times \cdots \times \Gamma_{n}$. It is easy to show that the homomorphism \[f \colon {L_{1}}^{\pi(S)} \mapsto p(S \cap (\Gamma_{1} \times \mathbb{Z}^{m_{2}} \times \cdots \times \mathbb{Z}^{m_{n}})) \slash p(S \cap (\mathbb{Z}^{m_{1}} \times \cdots \times \mathbb{Z}^{m_{n}}))\] defined by \[f(s)=(z_{2},\dots,z_{n}) p(S \cap (\mathbb{Z}^{m_{1}} \times \cdots \times \mathbb{Z}^{m_{n}})) \quad \text{if} \quad (sz_{1},z_{2},\dots,z_{n})\in S,\]
is well-defined, surjective and $\text{ker }f= \pi({L_{1}}^S)$. Therefore, by the First Isomorphism Theorem, we obtain the desired isomorphism. Note also that the group \[p(S \cap (\Gamma_{1} \times \mathbb{Z}^{m_{2}} \times \cdots \times \mathbb{Z}^{m_{n}})) \slash p(S \cap (\mathbb{Z}^{m_{1}} \times \cdots \times \mathbb{Z}^{m_{n}}))\] is an abelian group.

Finally, the elements of $\pi({L_{1}}^S)$ commute with the elements of ${L_{1}}^{\pi(S)} \slash \pi({L_{1}}^S)$, so it turns out that ${L_{1}}^{\pi(S)}= \pi({L_{1}}^S) \times A$, where $A$ is an abelian group. In particular, if ${L_{1}}^S$ is finitely generated, ${L_{1}}^{\pi(S)}$ is finitely generated.

For the converse, suppose that ${L_{1}}^{\pi(S)}$ is finitely generated. Then, $\pi({L_{1}}^S)$ is finitely generated and by taking into account the short exact sequence 
\[ 
\begin{tikzcd}
1 \arrow{r}  & {L_{1}}^{S} \cap \mathbb{Z}^{m_{1}+\dots +m_{n}} \arrow{r} & L_{1}^S \arrow{r} & \pi({L_{1}}^S) \arrow{r} & 1,
\end{tikzcd} 
\]
we have that ${L_{1}}^S$ is finitely generated if and only if $\pi({L_{1}}^S)$ is.

Thus, we can work with the subgroup $S^\prime= \pi(S) <\Gamma_{1}^\prime\times \cdots \times \Gamma_{n}^\prime.$ As in Proposition~\ref{Proposition 2}, we can assume that the projection maps $p_{i} \colon S^\prime \mapsto \Gamma_{i}^\prime$ are surjective and that the subgroups $L_{i}^{S^\prime}= S^\prime \cap \Gamma_{i}^\prime$ are non-trivial.

For $1 \leq i \leq r$, the subgroup $L_{i}^{S^\prime}$ is non-trivial, finitely generated and normal in $\Gamma_{i}^\prime$, so by Theorem~\ref{Theorem Section 5}, it has finite index in $\Gamma_{i}^\prime$. Let us define $p_{r}$ to be the homomorphism $ \Gamma_{1}^\prime\times \cdots \times \Gamma_{n}^\prime \mapsto \Gamma_{1}^\prime \times \cdots \times \Gamma_{r}^\prime$.

Since $\Delta= {L_{1}}^{S^\prime}\times \cdots \times {L_{r}}^{S^\prime}$ has finite index in $\Gamma_{1}^\prime \times \cdots \times \Gamma_{r}^\prime$, the subgroup $\hat{S_{0}}= {p_{r}}^{-1}(\Delta) \cap S^\prime$ has finite index in $S^\prime$. In addition, $\hat{S_{0}}= \Delta \times \hat{S_{2}}$, where \[\hat{S_{2}}= S^\prime \cap (\Gamma_{r+1}^\prime \times \cdots \times \Gamma_{n}^\prime).\]
Finally, by Theorem~\ref{Theorem C}, $\hat{S_{2}}$ has a finite index subgroup $S_{2}$ with $H_{k}(S_{2};\mathbb{Q})$ infinite dimensional for some $k\leq n-r$, so the same holds for $S$.

\section{Finitely presented residually Droms RAAGs}
\label{Section 10}

The goal of this section is to understand finitely presented residually Droms RAAGs. As we discussed in Section~\ref{Subsection}, these are precisely coordinate groups over Droms RAAGs, and coordinate groups are just finitely generated subgroups of direct products of limit groups. Thus, we focus on finitely presented subgroups of direct products of limit groups over Droms RAAGs. This section is based on the earlier work \cite{Bridson3} of Bridson, Howie, Miller and Short.

Let us prove Theorem~\ref{Theorem Section 10}:

\begin{theorem}\label{Theorem Section 10}
Let $S$ be a finitely generated group that is residually a Droms RAAG. The following are equivalent:\\[3pt]
(1) $S$ is finitely presentable;\\[3pt]
(2) $S$ is of type $FP_{2}(\mathbb{Q})$;\\[3pt]
(3) $\dim H_{2}(S_{0};\mathbb{Q})$ is finite for all subgroups $S_{0}<S$ of finite index;\\[3pt]
(4) There exists a neat embedding $S \hookrightarrow \Gamma_{0}\times \cdots \times \Gamma_{n}$ into a product of limit groups over Droms RAAGs such that the image of $S$ under the projection to $\Gamma_{i}\times \Gamma_{j}$ has finite index for $0\leq i<j \leq n$;\\[3pt]
(5) For every neat embedding $S \hookrightarrow \Gamma_{0}\times \cdots \times \Gamma_{n}$ into a product of limit groups over Droms RAAGs the image of $S$ under the projection to $\Gamma_{i}\times \Gamma_{j}$ has finite index for $0\leq i<j \leq n$.
\end{theorem}

\begin{proof}
We first check that for every finitely generated coordinate group over a Droms RAAG, there is a neat embedding.

Clearly, there is an embedding $S \hookrightarrow \Gamma_{0} \times \cdots \times \Gamma_{n}$ where $\Gamma_{i}$ is a limit group over a Droms RAAG. Each group $\Gamma_{i}$ is of the form $\mathbb{Z}^{m_{i}}\times \Gamma_{i}^\prime$ with $m_{i}\in \mathbb{N}\cup \{0\}$ and $\Gamma_{i}^\prime$ is a limit group over a Droms RAAG with trivial center. Therefore, we may assume that $S$ is a subgroup of the direct product $\Gamma_{0}\times \cdots \times \Gamma_{n}$ where $\Gamma_{0}$ is a free abelian group and $\Gamma_{i}$ is a limit group over a Droms RAAG with trivial center for $i>0$.

By the basis extension property for free abelian groups, there is a decomposition of $\Gamma_{0}$ as a direct sum $M_{0} \oplus R_{0}$ where $L_{0}= S \cap \Gamma_{0}$ has finite index in $M_{0}$. Since the intersection $S \cap R_{0}$ is trivial, the projection homomorphism $f \colon M_{0} \oplus R_{0} \times \Gamma_{1}\times \cdots \times \Gamma_{n} \mapsto M_{0} \times \Gamma_{1}\times \cdots \times \Gamma_{n}$ descends to a monomorphism $f_{| S} \colon S \mapsto f(S)$. Thus, $S$ is isomorphic to a subgroup of $M_{0} \times \Gamma_{1} \times \cdots \times \Gamma_{n}$ and we may now assume that $L_{0}$ has finite index in $\Gamma_{0}$. Finally, we may also suppose that $S$ is a full subdirect product as we did in Proposition~\ref{Proposition 2}.

Let us denote each projection map $\Gamma_{0}\times \cdots \times \Gamma_{n} \mapsto \Gamma_{i} \times \Gamma_{j}$ by $p_{ij}$, $0\leq i<j \leq n$.

The implications $(1) \implies (2)$, $(2) \implies (3)$ are clear. Since neat embeddings exist, $(5) \implies (4)$. By \cite[Theorem E]{Bridson3}, (4) implies (1). Thus it suffices to prove that (3) implies (5).

Let us argue first that for any $i>0$, $p_{0i}(S)$ has finite index in $\Gamma_{0}\times \Gamma_{i}$. Since $\Gamma_{0}$ is free abelian and the projection map $p_{i} \colon S \mapsto \Gamma_{i}$ is surjective, $p_{0i}(S)$ is isomorphic to $(p_{0i}(S) \cap \Gamma_{0}) \times \Gamma_{i}$. The subgroup $L_{0}$ is contained in $p_{0i}(S)\cap \Gamma_{0}$, so $p_{0i}(S)$ has finite index in $\Gamma_{0}\times \Gamma_{i}$.

Second, let us show that for $1\leq i < j \leq n$, $p_{ij}(S)$ has finite index in $\Gamma_{i}\times \Gamma_{j}$. There is an exact sequence
\[ 
\begin{tikzcd}
1 \arrow{r}  &\Gamma_{0} \arrow{r} & \Gamma_{0}\times \cdots \times \Gamma_{n} \arrow{r}{q} & \Gamma_{1}\times \cdots \times \Gamma_{n} \arrow{r} & 1,
\end{tikzcd} 
\]
and this exact sequence descends to another exact sequence for $S$:
\[ 
\begin{tikzcd}
1 \arrow{r}  &L_{0} \arrow{r} & S \arrow{r} & q(S) \arrow{r} & 1.
\end{tikzcd} 
\]
Thus, $S \slash L_{0} \cong q(S) < \Gamma_{1} \times \cdots \times \Gamma_{n}$.

Let us take $\overline{S_{0}}$ a finite index subgroup in $S \slash L_{0}$. Then, $\overline{S_{0}}$ is of the form $S_{0}L_{0} \slash L_{0}$ with $S_{0}L_{0}$ of finite index in $S$. The group $L_{0}$ is finitely generated and by hypothesis $\dim H_{2}(S_{0}L_{0}; \mathbb{Q})$ is finite.

Hence, $H_{2}(\overline{S_{0}}; \mathbb{Q})$ is finite dimensional for all finite index subgroups $\overline{S_{0}}$ in $S \slash L_{0}$. By Section~\ref{Section 6}, $N_{i,j}$ has finite index in $\Gamma_{j}$. But $N_{j,i}\times N_{i,j} < p_{ij}(S)$, so $p_{ij}(S)$ has finite index in $\Gamma_{i}\times \Gamma_{j}$.
\end{proof}

We now want to prove Theorem~\ref{Theorem end}. By Theorem~\ref{Theorem Section 10}, any subgroup of the direct product of limit groups over Droms RAAGs that contains a finitely presentable full subdirect product is again finitely presentable. Theorem~\ref{Theorem end} generalizes this.

\begin{theorem}\label{Theorem end}
Let $\Gamma_{1}\times \cdots \times \Gamma_{k}$ be the direct product of limit groups over Droms RAAGs, where $\Gamma_{1}$ is free abelian and $\Gamma_{i}$ is a limit group over a Droms RAAG with trivial center, for $i\in \{2,\dots,k\}$. Let $n\in \mathbb{N} \setminus \{1\}$, $S < \Gamma_{1}\times \cdots \times \Gamma_{k}$ be a full subdirect product and let $T < \Gamma_{1}\times \cdots \times \Gamma_{k}$ be a subgroup that contains $S$. If $S$ is of type $FP_{n}(\mathbb{Q})$, then so is $T$.
\end{theorem}

\begin{proof}

The quotient group $\Gamma_{1} \slash L_{1}$ is abelian, so in particular, nilpotent.

Moreover, if $S$ is of type $FP_{n}(\mathbb{Q})$ for $n\geq 2$, it is in particular of type $FP_{2}(\mathbb{Q})$. Then, by Theorem~\ref{Theorem Section 6}, $\Gamma_{i} /L_{i}$ is virtually nilpotent for $i\in \{2,\dots,k\}$. Thus, $D \slash L$ is virtually nilpotent, where $D=\Gamma_{1}\times \cdots \times \Gamma_{k}$ and $L= (S\cap \Gamma_{1})\times \cdots \times (S\cap \Gamma_{k})$.

By Corollary~\ref{Nilpotent 2}, there is a finite index subgroup $S_{0}< S$ and a subnormal chain \[S_{0}\triangleleft S_{1}\triangleleft \dots \triangleleft S_{l}=T,\] such that each quotient $S_{i+1}/S_{i}$ is either finite or infinite cyclic.

Since $S$ is of type $FP_{n}(\mathbb{Q})$ and $S_{0}$ has finite index in $S$, $S_{0}$ is also of type $FP_{n}(\mathbb{Q})$. Note that there is a short exact sequence \[ 1 \rightarrow S_{0} \rightarrow S_{1} \rightarrow S_{1}/S_{0} \rightarrow 1.\] Moreover, $S_{1}/S_{0}$ is of type $FP_{n}(\mathbb{Q})$ because it is infinite cyclic or a finite group. Therefore, $S_{1}$ is of type $FP_{n}(\mathbb{Q})$. By iterating this argument, we obtain that $T$ is of type $FP_{n}(\mathbb{Q})$.
\end{proof}

Finally, we focus on algorithmic problems. Let us start with the multiple conjugacy problem for finitely presented residually Droms RAAGs. 

The \emph{multiple conjugacy problem} for a finitely generated group $G$ asks if there is an algorithm that, given an integer $l$ and two $l$-tuples of elements of $G$, say $x=(x_{1},\dots,x_{l})$ and $y=(y_{1},\dots,y_{l})$, can determine if there exists $g\in G$ such that $g x_{i} g^{-1}= y_{i}$ in $G$, for $i\in \{1,\dots,l\}$.

When solving the multiple conjugacy problem for finitely presented residually free groups (see \cite{Bridson3}), the authors first show a result for bicombable groups. Fundamental groups of compact non-positively curved spaces are bicombable groups, so in particular, limit groups over Droms RAAGs (see \cite[Corollary 9.5]{Montse}).

\begin{proposition}{\cite[Proposition 7.5]{Bridson3}}
Let $\Gamma$ be a bicombable group, let $H < \Gamma$ be a subgroup, and suppose that there exists a subgroup $L<H$ normal in $\Gamma$ such that $\Gamma \slash L$ is nilpotent. Then $H$ has a solvable multiple conjugacy problem.
\end{proposition}

Second, they state a result that relates the decidability of the multiple conjugacy problem for a finite index subgroup and the whole group:

\begin{lemma}{\cite[Lemma 7.2]{Bridson3}}
Suppose $G$ is a group in which roots are unique and $H<G$ is a subgroup of finite index. If the multiple conjugacy problem for $H$ is solvable, then the multiple conjugacy problem for $G$ is solvable.
\end{lemma}

In order to apply this lemma, we need to check that limit groups over Droms RAAGs have unique roots. Recall that a group $G$ is said to have \emph{unique roots} if for all $x,y\in G$ and $n\neq 0$, one has $x=y \iff x^n=y^n$. In \cite{Duchamp} it is shown that RAAGs have unique roots. Hence it suffices to show that if $H$ has unique roots and $G$ is fully residually $H$, then $G$ has unique roots:

\begin{lemma}
Let $H$ be a group in which roots are unique and let $G$ be fully residually $H$. Then $G$ has unique roots.
\end{lemma}

\begin{proof}
Suppose that there are two elements $x\neq y\in G$ such that $x^n=y^n$. Since $G$ is fully residually $H$, there is a homomorphism $\phi \colon G \mapsto H$ such that $\phi(x)\neq \phi(y)$. However, $\phi(x)^n=\phi(y)^n$ and this contradicts the fact that $H$ has unique roots.
\end{proof}

Building on the previous results and Theorem~\ref{Theorem Section 6}, the proof of Theorem~\ref{Theorem Conjugacy} is the same as the proof of \cite[Theorem 7.4]{Bridson3}:

\begin{theorem}\label{Theorem Conjugacy}
The multiple conjugacy problem is solvable in every finitely presented group that is residually a Droms RAAG.
\end{theorem}

We conclude the paper by pointing out that the results that we have established for limit groups over Droms RAAGs allow us to apply the arguments of \cite{Separability} to prove the following results:

\begin{theorem}\label{Theorem Retract}
If $G$ is a finitely generated group that is residually a Droms RAAG and $H < G$ is a subgroup of type $FP_{\infty}(\mathbb{Q})$, then $H$ is a virtual retract of $G$.
\end{theorem}

\begin{theorem}\label{Theorem Separability}
If $G$ is a finitely generated group that is residually a Droms RAAG and $H$ is a finitely presentable subgroup of $G$ then $H$ is separable in $G$.
\end{theorem}

These two results are the analogous of \cite[Theorem A]{Separability} and \cite[Theorem B]{Separability}.

In \cite[Section 2]{Separability}, \cite[Theorem A]{Separability} and \cite[Theorem B]{Separability} are reduced to the case where $G$ is a direct product of limit groups over free groups and $H<G$ is a subdirect product. In our case, the same reduction can be applied since finitely generated residually Droms RAAGs embed as subdirect products in direct products of limit groups over Droms RAAGs. 

It is then shown, in \cite[Section 3]{Separability}, that \cite[Theorem A]{Separability} follows from \cite[Theorem A]{something1}. In our case, this is analogous to using Theorem~\ref{Theorem A}.

Finally, \cite[Theorem B]{Separability} requires \cite[Theorem 4.2]{something1}, which in the case of limit groups over Droms RAAGs corresponds to Theorem~\ref{Theorem Section 6}.

\end{document}